\documentclass[12pt,reqno]{amsart}
\usepackage{etoolbox}
\makeatletter
\patchcmd\maketitle
  {\uppercasenonmath\shorttitle}
  {}
  {}{}
\patchcmd\maketitle
  {\@nx\MakeUppercase{\the\toks@}}
  {\the\toks@}
  {}
  {}{}
\patchcmd\@settitle{\uppercasenonmath\@title}{\Large}{}{}
\patchcmd\@setauthors
  {\MakeUppercase{\authors}}
  {\authors}
  {}{}
\makeatother
\usepackage{amsmath,amssymb,amsthm}
\usepackage{xcolor}
\usepackage{url}
\usepackage{tikz-cd}
\usepackage[utf8]{inputenc}
\usepackage[T1]{fontenc}
\newtheorem{theorem}{Theorem}[section]
\newtheorem{definition}{Definition}[section]

\newtheorem{corollary}{Corollary}[section]
\newtheorem{proposition}{Proposition}[section]
\newtheorem{lemma}{Lemma}[section]
\newtheorem{remark}{Remark}[section]
\newtheorem{example}{Example}[section]

\numberwithin{equation}{section}

  \newcommand{\conv}{\mathrm{conv}}
  \newcommand{\R}{\mathfrak{Re}}
\usepackage[colorlinks=true]{hyperref}
\hypersetup{urlcolor=blue, citecolor=red , linkcolor= blue}
\usepackage[capitalise,noabbrev,nameinlink]{cleveref}      
\begin{document}
\author[C. Conde and K. Feki] {\large{Cristian Conde}$^{1_{a,b}}$ and \large{Kais Feki}$^{2}$}

\address{$^{[1_a]}$ Instituto Argentino de Matem\'atica ``Alberto Calder\'on", Saavedra 15 3er. piso, (C1083ACA), Buenos Aires, Argentina}
\address{$^{[1_b]}$ Instituto de Ciencias, Universidad Nacional de Gral. Sarmiento, J. M. Gutierrez 1150, (B1613GSX) Los Polvorines, Argentina}
\email{\url{cconde@campus.ungs.edu.ar}}

\address{$^{[2]}$ Department of Mathematics, Faculty of Arts and Sciences, Najran University, Najran 66445,  Kingdom of Saudi Arabia}
\email{\url{kfeki@nu.edu.sa}}

\date{\today}

\title[On approximate $A$-seminorm and $A$-numerical radius orthogonality]
{On approximate $A$-seminorm and $A$-numerical radius orthogonality of operators}
\maketitle

\begin{abstract}
This paper explores the concept of approximate Birkhoff-James orthogonality in the context of operators on semi-Hilbert spaces. These spaces are generated by positive semi-definite sesquilinear forms. We delve into the fundamental properties of this concept and provide several characterizations of it. Using innovative arguments, we extend a widely known result initially proposed by Magajna in [J. London. Math. Soc., 1993]. Additionally, we improve a recent result by Sen and Paul in [Math. Slovaca, 2023] regarding a characterization of approximate numerical radius orthogonality of two semi-Hilbert space operators, such that one of them is $A$-positive. Here, $A$ is assumed to be a positive semi-definite operator.
\end{abstract}

\section{Introduction and preliminaries} %
Let $\mathbb{B}(\mathcal{H})$ be the $C^{\ast}$-algebra of all bounded linear operators on a complex Hilbert space $\mathcal{H}$ with an inner product $\langle \cdot,\cdot \rangle$ and the corresponding norm $\|\cdot\| $. The symbol $I$ denotes the identity operator on $\mathcal{H}$. For any $T\in\mathbb{B}(\mathcal{H})$, the symbols $\mathcal{R}(T)$ and $\mathcal{N}(T)$ indicate the range and
the kernel of $T$, respectively. Throughout this article, the operator $A\in\mathbb{B}(\mathcal{H})$ is considered to be  non-zero and positive, i.e., $\langle Ax, x\rangle\geq 0$ for all $x\in\mathcal{H}$. Such an $A$ induces a positive semidefinite sesquilinear
form ${\langle \cdot, \cdot\rangle}_A: \,\mathcal{H}\times \mathcal{H} \rightarrow \mathbb{C}$
defined by
\begin{align*}
{\langle x, y\rangle}_A = \langle Ax, y\rangle, \qquad x, y\in\mathcal{H}.
\end{align*}
The symbol ${\|\cdot\|}_A$ signifies the seminorm induced by ${\langle \cdot, \cdot\rangle}_A$, i.e., ${\|x\|}_A = \sqrt{{\langle x, x\rangle}_A}$ for every $x\in\mathcal{H}$.  We denote the $A$-unit sphere of $\mathcal{H}$ by $\mathbb{S}_\mathcal{H}^A$, i.e.,
$$\mathbb{S}_\mathcal{H}^A:=\{x\in \mathcal{H}\,;\; \|x\|_A=1\}.$$
 Note that ${\|\cdot\|}_A$ is a norm on $\mathcal{H}$ if and only if the operator $A$ is injective. Further, it can be seen that
the semi-Hilbert space $(\mathcal{H}, {\|\cdot\|}_A)$ is complete if and only if $\mathcal{R}(A)$ is closed in $\mathcal{H}$.

For $x, y\in\mathcal{H}$, we say that $x$ and $y$ are $A$-orthogonal, denoted by $x \perp_A y$, if ${\langle x, y\rangle}_A = 0$.  The definition of $A$-orthogonality is a natural extension of
the usual notion of orthogonality, which in this context is the $I$-orthogonality.
\medskip

Given a positive semidefinite operator $A\in \mathbb{B}(\mathcal{H})$, we can replace the usual operator norm by the seminorm $\|\cdot\|_A$, where for any $T\in \mathbb{B}(\mathcal{H})$,
\[\|T\|_A=\sup\big\{\|Tx\|_A\,;\;~x\in \mathbb{S}_\mathcal{H}^A\big\}.\]
However, the quantity $\|\cdot\|_A$ is not sufficiently well-behaved. Specifically, it may happen that $\|T\|_A=+\infty$, for some $T\in \mathbb{B}(\mathcal{H})$. On the other hand, we do not have any obvious choice for an  adjoint operation defined by $A$. To conduct a meaningful study of orthogonality in $\mathbb{B}(\mathcal{H})$ with respect to the seminorm $\|\cdot\|_A$, we consider the following collection:
$$\mathbb{B}_{A^{1/2}}(\mathcal{H}) = \Big\{T\in \mathbb{B}(\mathcal{H})\, : \,\, \exists\,\, c>0\,\text{ such that }
\,\,{\|Tx\|}_{A} \le  c{\|x\|}_{A}, \,\, \forall x\in \mathcal{H}\Big\}.$$
An operator $T\in \mathbb{B}(\mathcal{H})$ is said to be $A$-bounded if $T\in\mathbb{B}_{A^{1/2}}(\mathcal{H})$. One can show that $\mathbb{B}_{A^{1/2}}(\mathcal{H})$ is  a unital subalgebra of $\mathbb{B}(\mathcal{H})$ which in general is neither closed nor dense in $\mathbb{B}(\mathcal{H})$ (see \cite{acg2,feki01}).
The subalgebra $\mathbb{B}_{A^{1/2}}(\mathcal{H})$ is equipped with the following seminorm:
\begin{align*}
{\|T\|}_A = \displaystyle{\sup_{x \in \overline{\mathcal{R}(A)}, x\neq \mathbf{0}}} \frac{{\|Tx\|}_A}{{\|x\|}_A}
= \inf\Big\{c>0\,; \; {\|Tx\|}_A \leq c{\|x\|}_A,\;\forall\, x\in \mathcal{H}\Big\}< \infty.
\end{align*}
For any $T\in \mathbb{B}_{A^{1/2}}$, the quantity $\|T\|_A$ can be equivalently expressed as:
\begin{align*}
{\|T\|}_A = \displaystyle\sup_{x\in \mathbb{S}_\mathcal{H}^A} {\|Tx\|}_A
= \sup\Big\{|{\langle Tx, y\rangle}_A|\,; \; x, y\in \mathbb{S}_\mathcal{H}^A\Big\}.
\end{align*}
For a comprehensive study on this classes of operators, the readers are referred to \cite{acg1,acg2,acg3,feki01}. Saddi in \cite{saddi} defined the so-called $A$-numerical radius of an operator $T\in\mathbb{B}(\mathcal{H})$ by
\begin{align*}
\omega_A(T)
&:= \sup \left\{|\langle Tx, x\rangle_A|\,;\;x\in \mathbb{S}_\mathcal{H}^A\right\}.
\end{align*}
 It is worth mentioning that $\omega_A(T)$ and ${\|T\|}_A$ can be $+ \infty$ for some $T\in\mathbb{B}(\mathcal{H})$ (see \cite{feki01}). However, ${\|\cdot\|}_A$ and $\omega_A(\cdot)$ define two equivalent seminorms on $\mathbb{B}_{A^{1/2}}(\mathcal{H})$. More precisely, for all $T\in \mathbb{B}_{A^{1/2}}(\mathcal{H})$, we have
\begin{equation}\label{refine1}
\frac{1}{2} \|T\|_A\leq\omega_A(T) \leq \|T\|_A,
\end{equation}
(see \cite{feki01} and references therein for more details). Recall from \cite{feki01} that an operator $T\in \mathbb{B}_{A^{1/2}}(\mathcal{H})$ is called $A$-normaloid if $r_A(T)=\|T\|_A$ where $r_A(T)$ means the $A$-spectral radius of $T$ (see \cite{feki01}). It is shown in \cite{feki01} that the equality $\omega_A(T)=\|T\|_A$ holds when $T$ is an $A$-normaloid operator. Furthermore, if $AT^2=0$, then $\frac{1}{2} \|T\|_A=\omega_A(T)$ (see \cite{feki01}). In recent years, the study of operators  on semi-Hilbert spaces received a considerable attention of many authors
(see, e.g., \cite{acg2,bakna2,bdd,gu,ssp,ssp2} and the references therein).
\medskip

There are several notions of orthogonality in $\mathbb{B}(\mathcal{H})$ (see, e.g., \cite{M.Z}). However, among the types, Birkhoff-James orthogonality is considered to be the most useful due to its diverse application in understanding the geometry of the operator spaces. Given $T, S\in \mathbb{B}(\mathcal{H})$, we say that $T$ is orthogonal to $S$, in the sense of Birkhoff-James, symbolized as $T\perp^B S$, if
\begin{align*}
\|T + \lambda S\|\geq \|T\|\quad \mbox{for all} \,\, \lambda \in \mathbb{C}.
\end{align*}
 In 1993, Magajna obtained the first characterization of the Birkhoff-James orthogonality (see \cite[Lemma 2.2]{Mag}). More precisely, given any $T, S\in \mathbb{B}(\mathcal{H})$, $T\perp^B S$ if and only if there exists
a sequence of unit vectors $\{x_n\}$ in $\mathcal{H}$ such that
\begin{align*}
\lim_{n\rightarrow\infty} \|Tx_n\| = \|T\| \quad
\mbox{and} \quad \lim_{n\rightarrow\infty}\langle Tx_n, Sx_n\rangle = 0.
\end{align*}
Therefore, whenever $\mathcal{H}$ is finite-dimensional,
$T\perp^B S$ if and only if there is a unit vector $x\in\mathcal{H}$ such that
$\|Tx\| = \|T\|$ and $\langle Tx, Sx\rangle = 0$.
Several years later, more precisely in 1999, Bhatia and \v{S}emrl \cite[Remark 3.1]{B.S} and Paul \cite[Lemma 2]{Pa} proved independently the same statement.

\medskip

 In recent times, the result due to Bhatia and \v{S}emrl has been extended in different contexts (see \cite{Ch.St.Wo, Wo.3, Wo.1}) employing different methods \cite{Wo.3, Wo.1}. This facilitates studying various aspects of orthogonality of
operators acting on Banach spaces and the settings of Hilbert $C^*$-module;
 see, for instance, \cite{A.R, B.C.M.W.Z, Ch.Wo, G.S.P, Ke, P.S.M.M, Wo.2}.
\medskip

The notion of $A$-Birkhoff-James orthogonality of operators has been recently given by Zamani in \cite{za} as follows.
\begin{definition}\label{de.31}
Let $T,S \in \mathbb{B}_{A^{1/2}}(\mathcal{H})$. The operator $T$ is called an $A$-Birkhoff--James orthogonal
to $S$, denoted by $T\perp^B_A S$, if
\begin{align*}
{\|T + \lambda S\|}_A\geq {\|T\|}_A \quad \forall\, \lambda \in \mathbb{C}.
\end{align*}
\end{definition}
 For $A=I$, the above definition reduces to the well-known notion of Birkhoff--James orthogonality of Hilbert space operators. Note that $A$-Birkhoff--James orthogonality is homogenous, i.e., $T\perp^B_A S\, \Leftrightarrow \, (\alpha T)\perp^B_A (\beta S)$, for every $\alpha, \beta\in\mathbb{C}$.
\medskip
Recently, Sen et al. introduced in \cite{ssp} the notion of approximate orthogonality  with respect to the seminorm $\|\cdot\|_A$ as follows.
\begin{definition}\label{appnorm}
Let $T,S \in \mathbb{B}_{A^{1/2}}(\mathcal{H})$ and $\varepsilon \in [0,1).$  Then $T$ is said to be $(\varepsilon,A)$-approximate orthogonal to $S \in \mathbb{B}_{A^{1/2}}(\mathcal{H}),$  written as $T\bot^B_{A,\varepsilon}S,$ if
$$\|T+\lambda S\|_A^2 \geq \|T\|_A^2 -2 \varepsilon \|T\|_A \|\lambda S\|_A,$$
for all $\lambda \in \mathbb{C}.$
 \end{definition}
From the definition of $(\varepsilon,A)$-approximate orthogonality, it is easy to see that such orthogonality notion is homogenous.  Below, we furnish a supportive example regarding this newly introduced concept of orthogonality.
\begin{example}\label{example1}
Let $\mathcal{H}=l^2(\mathbb{N})$ and we consider the following linear operators defined on $\mathcal{H}.$
\begin{enumerate}
			\item $A(x_1, x_2,\cdots, x_n, \cdots)=(x_1, x_2, 0, 0,\cdots, 0, 0, 0, \cdots),$
			\item $Te_1=\alpha e_1$ and $Te_j=\alpha_{j-1}e_j$, where $\{e_j\}_{j\in \mathbb{N}}$ is the canonical basis of $\mathcal{H}$ and $0<\alpha_1<\cdots<\alpha,$
			\item $S(x_1, x_2,\cdots, x_n, \cdots)=(\varepsilon x_1, x_2, x_3,\cdots, x_n, x_{n+1}, \cdots),$  where $\varepsilon \in [0, 1).$
\end{enumerate}
 Obviously, $A$ is a positive bounded operator on $\mathcal{H},$ $\|T\|_A=\alpha=\|Te_1\|_A$  and $\|S\|_A=1.$ \\
Then, for any $\lambda=|\lambda|e^{i\alpha}\in \mathbb{C}$, we have
	\begin{eqnarray}
	\|T+\lambda S\|_A^2&\geq &\|(T+\lambda S)e_1\|_A^2\nonumber \\
	&=&\|Te_1\|_A^2+2|\lambda|\R(e^{-i\alpha}\langle Te_1, Se_1 \rangle_A)+|\lambda|^2\|Se_1\|_A^2\nonumber \\
	&\geq& \|Te_1\|_A^2+2|\lambda|\R(e^{-i\alpha}\langle Te_1, Se_1 \rangle_A)\nonumber\\
		&=& \|Te_1\|_A^2+2|\lambda|\varepsilon \alpha \cos(\alpha)\nonumber\\
	&\geq& \|Te_1\|_A^2-2|\lambda| \varepsilon \alpha \nonumber\\
		&=& \|T\|_A^2-2\varepsilon \|T\|_A\|\lambda S\|_A. \nonumber\
	\end{eqnarray}
	So, by Definition \ref{appnorm}, we conclude that  $T\bot^B_{A,\varepsilon} S.$
\end{example}

\medskip

Let $T\in \mathbb{B}(\mathcal{H})$. An operator $R\in \mathbb{B}(\mathcal{H})$ is called an $A$-adjoint operator of $T$ if
the following equality ${\langle Tx, y\rangle}_A = {\langle x, Ry\rangle}_A$ holds for all $x, y\in \mathcal{H}$, that is, $AR = T^*A$. Hence, the existence of an $A$-adjoint of $T$ is equivalent to the existence of solutions in $\mathbb{B}(\mathcal{H})$ of the equation $AX=T^*A$. It can be observed that neither the existence nor the uniqueness of solutions of the above equation is ensured. However, the well-known theorem due to Douglas \cite{doug} guarantees that if $T\in \mathbb{B}(\mathcal{H})$ and satisfies the property $\mathcal{R}(T^*A) \subseteq \mathcal{R}(A),$ then the equation $AX = T^*A$ admits at least one solution in $\mathbb{B}(\mathcal{H})$. Again Douglas Theorem ensures that if the operator equation $AX = T^*A$ has solution, then the above equation admits a unique solution which verifies $\mathcal{R}(T^{\sharp_A})\subseteq \overline{\mathcal{R}(A)}$. This unique solution, denoted by $T^{\sharp_A}$, is said the Douglas solution or the reduced solution of $AX=T^*A$. Let $\mathbb{B}_A(\mathcal{H})$ stands for the collection of all operators which admit $A$-adjoint operators, i.e.,
$$ \mathbb{B}_A(\mathcal{H}) = \{ T \in \mathbb{B}(\mathcal{H}) \,;\, \mathcal{R}(T^*A) \subseteq \mathcal{R}(A) \}.$$
An operator $T\in \mathbb{B}(\mathcal{H})$ is called $A$-positive if the operator $AT$ is positive and we note $T\geq_A0$. It is not difficult to check that $A$-positive operators are in $\mathbb{B}_A(\mathcal{H})$. Another application Douglas theorem \cite{doug} shows that the set of all operators which admit $A^{1/2}$-adjoint operators coincides with $\mathbb{B}_{A^{1/2}}(\mathcal{H})$. We mention that $\mathbb{B}_{A}(\mathcal{H})$ and $\mathbb{B}_{A^{1/2}}(\mathcal{H})$ represent two subalgebras of $\mathbb{B}(\mathcal{H})$ which are, in general, neither closed nor dense in $\mathbb{B}(\mathcal{H})$. Further, we have $\mathbb{B}_{A}(\mathcal{H}) \subseteq \mathbb{B}_{A^{1/2}}(\mathcal{H}) \subseteq \mathbb{B}(\mathcal{H}) $. Generally, the above inclusions are proper with equality if $A$ is one-to-one and has a closed range (see  \cite{acg1, feki01}). Now, let $T\in\mathbb{B}_{A}(\mathcal{H})$. Then, $T^{\sharp_A}\in\mathbb{B}_{A}(\mathcal{H})$ and
$(T^{\sharp_A})^{\sharp_A} = P_ATP_A$. Here $P_A$ denotes the orthogonal projection onto $\overline{\mathcal{R}(A)}$, the closure of $\mathcal{R}(A)$. For an account of results, the  readers are referred to \cite{acg1,acg2,acg3,feki01}.

The article consists of three sections, including an introduction. In the following section, we introduce the concept of $A$-approximate orthogonality associated with the seminorm $\|\cdot\|_A$ and explore some basic properties. Furthermore, we present a comprehensive characterization of $A$-approximate orthogonality, extending a result previously established by different authors within the context of the usual norm in $\mathbb{B}(\mathcal{H})$. In Section \ref{s3}, we address the concept of approximate orthogonality induced by $A$-numerical radius, denoted as $\bot_{\omega_A}^\epsilon$, recently introduced in \cite{sen}. Essentially, we build upon a recent result by Sen and Paul in \cite{sen} concerning the characterization of $T \bot_{\omega_A}^\epsilon S$ under the conditions $T\geq_A 0$ and $S$ is $A$-bounded.

\section{Approximate $A$-Birkhoff-James Orthogonality of operators}\label{s2}

In the following proposition, we state some properties of the notion of approximate $A$-Birkhoff-James orthogonality of operators.
\begin{proposition}\label{properties}
Let $T, S \in \mathbb{B}_{A^{1/2}}(\mathcal{H})$ and $\varepsilon \in [0,1).$ Then the following conditions are equivalent
\begin{itemize}
  \item [(1)] $T\bot^B_{A,\varepsilon}S.$
  \item [(2)] $T\bot^B_{A,\varepsilon}P_AS.$
  \item [(3)] $P_AT\bot^B_{A,\varepsilon}P_AS.$
\end{itemize}
In addition, if $T, S\in \mathbb{B}_{A}(\mathcal{H})$ then any of the previous  conditions is equivalent to
\begin{itemize}
  \item [(4)] $T^{\sharp_A}\bot^B_{A,\varepsilon}S^{\sharp_A}.$
\end{itemize}
\end{proposition}
\begin{proof}
 Let $\lambda \in \mathbb{C}$, then
	\begin{align*}
&\|P_AT+\lambda P_AS\|_A\nonumber\\
&=\sup\Big\{|{\langle (P_AT+\lambda P_AS)x, y\rangle}_A|\,; \; x, y\in \mathcal{H}, {\|x\|}_A = {\|y\|}_A = 1\Big\}\nonumber \\
	&=\sup\Big\{|{\langle AP_ATx, y\rangle}+{\lambda\langle  AP_ASx, y\rangle}|\,; \; x, y\in \mathcal{H}, {\|x\|}_A = {\|y\|}_A = 1\Big\}\nonumber \\
	&=\sup\Big\{|{\langle Tx, P_AAy\rangle}+{\lambda\langle  Sx, P_A Ay\rangle}|\,; \; x, y\in \mathcal{H}, {\|x\|}_A = {\|y\|}_A = 1\Big\}\nonumber \\
	&=\sup\Big\{|{\langle Tx, Ay\rangle}+{\lambda\langle  Sx,  Ay\rangle}|\,; \; x, y\in \mathcal{H}, {\|x\|}_A = {\|y\|}_A = 1\Big\}\nonumber \\
	&=\sup\Big\{|{\langle ATx, y\rangle}+{\lambda\langle  ASx,  y\rangle}|\,; \; x, y\in \mathcal{H}, {\|x\|}_A = {\|y\|}_A = 1\Big\}\nonumber \\
	&=\sup\Big\{|{\langle A(T+\lambda S)x, y\rangle}|\,; \; x, y\in \mathcal{H}, {\|x\|}_A = {\|y\|}_A = 1\Big\}\nonumber \\
	&=\|T+\lambda S\|_A,
	\end{align*}
	in a similar way we obtain that
	\begin{align}\label{eq3}
	\|T+\lambda P_AS\|_A=\|T+\lambda S\|_A
	\end{align}
	and
\begin{align}\label{eq2}
\|P_AS\|_A
&=\sup\Big\{|{\langle P_ASx, y\rangle}_A|\,; \; x, y\in \mathcal{H}, {\|x\|}_A = {\|y\|}_A = 1\Big\}\nonumber \\
&=\sup\Big\{|{\langle AP_ASx, y\rangle}|\,; \; x, y\in \mathcal{H}, {\|x\|}_A = {\|y\|}_A = 1\Big\}\nonumber \\
&=\sup\Big\{|{\langle Sx, P_A Ay\rangle}|\,; \; x, y\in \mathcal{H}, {\|x\|}_A = {\|y\|}_A = 1\Big\}\nonumber \\
&=\sup\Big\{|{\langle Sx,  Ay\rangle}|\,; \; x, y\in \mathcal{H}, {\|x\|}_A = {\|y\|}_A = 1\Big\}=\|S\|_A. \
\end{align}
\noindent$(1) \Leftrightarrow (2):$  Follows directly from Definition \ref{appnorm} and the equalities \eqref{eq3} and \eqref{eq2}.\\
\noindent$(1) \Leftrightarrow (3):$  Follows directly from Definition \ref{appnorm} and the equalities \eqref{eq3} and \eqref{eq2}.\\
\noindent$(1) \Leftrightarrow (4):$
 Since $\|R\|_A=\|R^{\sharp_A}\|_A$ for any $R\in \mathbb{B}_{A}(\mathcal{H})$, this equivalence follows easily from some well-known properties of $T^{\sharp_A}$, recall that $\|R\|_A=\|R^{\sharp_A}\|_A$ for any $R\in \mathbb{B}_{A}(\mathcal{H}).$
\end{proof}

\begin{corollary}
Let $T, S\in \mathbb{B}_{A}(\mathcal{H})$ such that $T^{\sharp_A}\bot^B_{A,\varepsilon}S^{\sharp_A},$ then $TP_A\bot^B_{A,\varepsilon}SP_A.$
\end{corollary}
\begin{proof}
 Since $T^{\sharp_A}\bot^B_{A,\varepsilon}S^{\sharp_A}$, it follows from Proposition \ref{properties} that $(T^{\sharp_A})^{\sharp_A}\bot^B_{A,\varepsilon}(S^{\sharp_A})^{\sharp_A}.$ Thus, $P_ATP_A \bot^B_{A,\varepsilon} P_ASP_A.$ However, since $TP_A, SP_A\in \mathbb{B}_{A}(\mathcal{H})\subseteq \mathbb{B}_{A^{1/2}}(\mathcal{H})$, by Proposition \ref{properties}, we get $TP_A\bot^B_{A,\varepsilon}SP_A.$
\end{proof}

\begin{proposition} \label{linearlyind}
Let $T, S\in \mathbb{B}_{A^{1/2}}(\mathcal{H})$  such that $T\bot^B_{A,\varepsilon}S,$ for some $\varepsilon \in [0,1)$. Then $T, S$ are linearly independent.
\end{proposition}
\begin{proof}
Suppose that $T$ and $S$ are not linearly independent, this means that there exists $\alpha \in \mathbb{C}, \alpha\neq 0$ such that $T=\alpha S.$
	
	We divide the proof into two main cases, namely $\varepsilon=0$ and $\varepsilon >0$.
	
	Case 1: $\varepsilon=0.$
	
 It follows from the Definition of \ref{appnorm} that $\|T+\lambda S\|_A^2 \geq \|T\|_A^2$ for any scalar $\lambda\in \mathbb{C}$. Consequently,
	\begin{align*}
	\|(\alpha+\lambda)S\|_A^2=|\alpha+\lambda|^2\|S\|_A^2=\|T+\lambda S\|_A^2\geq \|T\|_A^2=|\alpha|^2\|S\|_A^2.
	\end{align*}
Specifically, for $\lambda=-\alpha\neq 0$, we obtain $0\geq |\alpha|^2\|S\|_A^2,$ which is a contradiction.
	
	Case 2: $\varepsilon\in (0, 1).$
	
 Alike of Case 1, the definition of $(\varepsilon,A)$-approximate produces
	 \begin{align*}
	 |\alpha+\lambda|^2\|S\|_A^2=\|T+\lambda S\|_A^2\geq \|T\|_A^2 -2 \varepsilon \|T\|_A \|\lambda S\|_A=(|\alpha|^2-2\varepsilon|\lambda||\alpha|)\|S\|_A^2,
	 \end{align*}
for all $\lambda \in \mathbb{C}$. Since $\|S\|_A\neq 0$, we have
	\begin{align*}
	|\alpha+\lambda|^2\geq (|\alpha|^2-2\varepsilon|\lambda||\alpha|).
	\end{align*}
Considering the sequence ${\frac{-\alpha}{2^n}}$ in place of $\lambda$ in the above inequality, we have
	 \begin{align*}
	 \left(1-\frac{1}{2^n}\right)^2|\alpha|^2\geq \left(1-\frac{\varepsilon}{2^{n-1}}\right)|\alpha|^2.
	 \end{align*}
Consequently,
	  \begin{align*}
	 \varepsilon\geq \left(1-\left(1-\frac{1}{2^n}\right)^2\right)2^{n-1}=1-\frac{1}{2^{n+1}}=a_n, \qquad n\in \mathbb{N}.
	 \end{align*}
However, since $\lim\limits_{n\to \infty}a_n=1$, we get $\varepsilon\geq 1$, a contradiction.
	
Thus, from both cases, we conclude that $T$ and $S$ are linearly independent.
\end{proof}

\begin{proposition}\label{rangesort}
Let $T, S \in \mathbb{B}_{A^{1/2}}(\mathcal{H})$ be such that
\begin{equation}\label{5rra}
 \|Tx+\lambda Sy\|_A^2\geq \|Tx\|_A^2-2\varepsilon |\lambda|\|Tx\|_A\|Sy\|_A,
\end{equation}
  for any $x, y\in \mathcal{H}$, $\lambda \in \mathbb{C}$ and for some $\varepsilon \in [0,1).$ Then $T\bot^B_{A,\varepsilon}S$.
\end{proposition}
\begin{proof}
	Let $x\in \mathcal{H}$ with $\|x\|_A=1$ and $\lambda\in \mathbb{C}$, then by taking \eqref{5rra} into account, we obtain
	\begin{align}\label{ineq5}
	\|(T+\lambda S)x\|_A^2\geq \|Tx\|_A^2-2\varepsilon |\lambda|\|Tx\|_A\|Sx\|_A.
	\end{align}
Moreover, since $\|Tx\|_A\leq \|T\|_A$ and $\|Sx\|_A\leq \|S\|_A$, then \eqref{ineq5} yields that
	\begin{align*}
		\|(T+\lambda S)x\|_A^2\geq \|Tx\|_A^2-2\varepsilon |\lambda|\|T\|_A\|S\|_A.
	\end{align*}
	Taking the supremum over all  $x\in \mathcal{H}$ such that $\|x\|_A=1$ in the last inequality, we get
	\begin{align*}
	\|T+\lambda S\|_A^2\geq \|T\|_A^2-2\varepsilon |\lambda|\|T\|_A\|S\|_A,
	\end{align*}
	for any $\lambda \in \mathbb{C}.$ This means that $ T\bot^B_{A,\varepsilon}S.$
\end{proof}	
	
Given any $T\in \mathbb{B}_{A^{1/2}}(\mathcal{H})$, we denote by $\mathbb{M}_T^A$
the set of all $A$-unit vectors at which $T$ attains its norm, i.e.,
$\mathbb{M}_T^A=\big\{x\in \mathbb{S}_\mathcal{H}^A;\,\, \|Tx\|_A = \|T\|_A\big\}.$ In particular, if $x_0\in  \mathbb{M}_T^A$ such that $Tx_0 \bot^B_{A,\varepsilon} Sx_0$ for some $\varepsilon \in [0, 1)$, then for any $\lambda \in \mathbb{C}$ we have
\begin{align*}
\|(T+\lambda S)\|_A^2&\geq \|(T+\lambda S)x_0\|_A^2\geq \|Tx_0\|_A^2-2\varepsilon |\lambda|\|Tx_0\|_A\|Sx_0\|_A\nonumber\\
&\geq \|T\|_A^2-2\varepsilon |\lambda|\|T\|_A\|S\|_A,
\end{align*}
 i.e., $T\bot^B_{A,\varepsilon}S$. In conclusion, the hypothesis in Proposition \ref{rangesort} can be significantly weakened if $\mathbb{M}_T^A\neq \emptyset$. We record this observation as a proposition below.
\begin{proposition}
Let $T, S \in \mathbb{B}_{A^{1/2}}(\mathcal{H})$ be such that
\begin{equation}\label{5rra}
\|Tx_0+\lambda Sx_0\|_A^2\geq \|Tx_0\|_A^2-2\varepsilon |\lambda|\|Tx_0\|_A\|Sx_0\|_A,
\end{equation}
for any $x_0\in   \mathbb{M}_T^A$ and for some $\varepsilon \in [0,1).$ Then $T\bot^B_{A,\varepsilon}S$.
	\end{proposition}

 We need two technical lemmas in the sequel.

\begin{lemma}\cite{za}\label{Lemma: 1}
Let $T,S\in \mathbb{B}_{A^{1/2}}(\mathcal{H})$. Then, the set
\begin{align*}\label{Lambda}
W_A(T, S):=\big\{\lambda\in \mathbb{C}:~\langle Sx_n, Tx_n \rangle_A \to \lambda,~\{x_n\}\subseteq \mathbb{S}_\mathcal{H}^A,~\|Tx_n\|_A\to \|T\|_A\big\}.
\end{align*} is nonempty, convex and
compact subset of $\mathbb{C}$.
\end{lemma}
%

For any non-empty subset $\Lambda$ of $\mathbb{C}$, denote the convex hull of $\Lambda$ by $\conv(\Lambda)$. For any non-zero complex number $\mu$, $\mathsf{sgn}(\mu)=\frac{\mu}{|\mu|}$.

\begin{lemma}\label{Lemma: 2}
Let $\Lambda$ be a non-empty compact subset of $\mathbb{C}$ with the property that for any $\mu \in S^1$ there exists $\lambda \in \Lambda$ such that $\R~\mu \lambda \geq -r$ for some positive real number $r$. Then $\conv(\Lambda)\cap B(0,r)\neq \emptyset$, where $B(0,r)$ denotes the closed disc of radius $r$, centred at $0$.
\end{lemma}
\begin{proof}
Suppose by contradiction $\conv(\Lambda)\cap B(0,r)=\emptyset$. Let $\Omega$ be a subset of the cube $[0,1]^3$, defined by
\[\Omega:=\{(t_1,t_2,t_3)\in [0,1]^3:~t_1+t_2+t_3=1\}.\]
Note that
\[\conv(\Lambda)=\Phi\left(\Lambda^3\times \Omega\right),\]
where $\Phi:\left(\Lambda^3\times \Omega\right)\to \mathbb{C}$ is a function, defined by
\[\Phi((\lambda_1,\lambda_2, \lambda_3),(t_1,t_2,t_3))=\sum\limits_{k=1}^3 \lambda_k t_k.\]
Since $\Lambda$ and $\Omega$ are compact, so is $\left(\Lambda^3\times \Omega\right)$. Moreover, $\conv(\Lambda)$ is compact, since $\Phi$ is continuous.
\medskip

Now, $\conv(\Lambda)$ is a compact convex subset of $\mathbb{C}$ which does not contain $0$. Therefore, $\conv(\Lambda)$ possesses a unique closest point to the origin, say $|d|e^{i\alpha}$, for some $\alpha\in [0,2\pi)$. Note that $|d|> r$, as $\conv(\Lambda)\cap B(0,r)=\emptyset$. However, this shows that
\[\R~e^{i(\pi-\alpha)}\lambda < -r,\qquad \lambda\in \conv(\Lambda)\supseteq \Lambda.\]
Therefore, we arrive at a contradiction, and the proof follows.
\end{proof}
In \cite[Theorem 3.2 and 3.3]{Ch.St.Wo}, Chmielinski et al. characterized the approximate orthogonality in the class of bounded linear operators on a real Hilbert space.  Our next result provides a complete  characterization of the $(\varepsilon,A)$-approximate orthogonality in $\mathbb{B}_{A^{1/2}}(\mathcal{H})$. 

\begin{theorem}\label{Theorem: 1}
Let $T,S\in \mathbb{B}_{A^{1/2}}(\mathcal{H})$. Then for any given $\varepsilon\in [0,1)$, the following conditions are equivalent:
\begin{itemize}
    \item [(i)] $T\perp_{A,\varepsilon}^B S$.
    \item[(ii)] $W_A(T, S)\cap B(0,\varepsilon\|T\|_A\|S\|_A)\neq \emptyset$.
\end{itemize}
\end{theorem}
\begin{proof}
(i)$\implies$ (ii): Let $\mu$ be any unimodular constant. Since $T\perp_{A,\varepsilon}^B S$, we have
\[\left\|T+\frac{\mu}{n}S\right\|_A^2> \|T\|_A^2+\frac{1}{n^2}\|S\|_A^2-\frac{2}{n}\varepsilon \|T\|_A\|S\|_A-\frac{1}{n^2}, \qquad n\in \mathbb{N}.\]
In fact, for each $n\in \mathbb{N}$, we can find $x_n\in \mathbb{S}_{\mathcal{H}}^A$ such that
\[\left\|\left(T+\frac{\mu}{n}S\right)x_n\right\|_A^2> \|T\|_A^2+\frac{1}{n^2}\|S\|_A^2-\frac{2}{n}\varepsilon \|T\|_A\|S\|_A-\frac{1}{n^2}.\]
On simplification,
\[\|Tx_n\|_A^2>\|T\|_A^2-\frac{1}{n^2}(\|S\|_A^2+1)-\frac{2}{n}\left(\R~\mu\left\langle Sx_n, Tx_n\right\rangle_A-\varepsilon \|T\|_A\|S\|_A\right).\]
On the other hand,
\[\R~\mu\left\langle Sx_n, Tx_n\right\rangle_A > \frac{n}{2}(\|Tx_n\|_A^2-\|T\|_A^2)-\frac{1}{2n}(\|S\|_A^2+1)-\varepsilon \|T\|_A\|S\|_A.\]
Since both $\{\|Tx_n\|_A\}$ and $\{\left\langle Sx_n, Tx_n\right\rangle_A\}$ are bounded sequences in $\mathbb{C}$, we can find a common subsequence $\{n_k\}$ of natural numbers such that both $\left\{\left\|Tx_{n_k}\right\|_A\right\}$ and $\left\{\left\langle Sx_{n_k}, Tx_{n_k}\right\rangle_A\right\}$ are convergent. However, since
\[\|T\|_A\geq \lim\limits_{k\to \infty}\left\|Tx_{n_k}\right\|_A\geq \|T\|_A,\]
we get $\left\|Tx_{n_k}\right\|_A\to \|T\|_A.$ Moreover, since $\{x_{n_k}\}\subseteq \mathbb{S}_{\mathcal{H}}^A$,  $\lim\limits_{k\to \infty} \left\langle Sx_{n_k}, Tx_{n_k}\right\rangle_A\in W_A(T, S)$. Also, note that
\[\R~\mu  \lim\limits_{k\to \infty}\left\langle Sx_{n_k}, Tx_{n_k}\right\rangle_A \geq -\varepsilon \|T\|_A\|S\|_A.\]
Therefore, by Lemma \ref{Lemma: 2}, we get
\[W_A(T, S)\cap B(0,\varepsilon\|T\|_A\|S\|_A)\neq \emptyset.\]

\medskip

(ii) $\implies$ (i): Let $\mu \in \mathbb{C}\setminus \{0\}$ be arbitrary. We show that there exists $\lambda\in W_A(T, S)$ such that $\R~\mathsf{sgn}(\mu) \lambda \geq -\varepsilon\|T\|_A\|S\|_A.$

\medskip

Clearly, there exists a member $\alpha$ of $W_A(T, S)$ such that \[|\R~\alpha|\leq |\alpha|\leq\varepsilon\|T\|_A\|S\|_A.\]
By the Caratheodory Theorem, there  exist $\lambda_i\in W_A(T, S)$ and $t_i\in [0,1]$; $1\leq i \leq 3$ such that
\[\sum_{i=1}^3t_i\lambda_i=\alpha, ~ \sum\limits_{i=1}^3 t_i =1.\]
Now,
\[\sum_{i=1}^3t_i\R~\mathsf{sgn}(\mu)\lambda_i=\R~\mathsf{sgn}(\mu)\alpha.\]
If $\R~\mathsf{sgn}(\mu)\lambda_i < -\varepsilon\|T\|_A\|S\|_A$ for all $1\leq i \leq 3$, then
\[\R~\mathsf{sgn}(\mu)\alpha=\sum_{i=1}^3t_i\R~\mathsf{sgn}(\mu)\lambda_i<-\varepsilon\|T\|_A\|S\|_A.\]
However, this is a contradiction, as
\[|\R~\mathsf{sgn}(\mu) \alpha|\leq |\mathsf{sgn}(\mu)\alpha| \leq \varepsilon\|T\|_A\|S\|_A.\]
Therefore, at least for one $i\in \{1,2,3\}$, we must have
\[\R~\mathsf{sgn}(\mu) \lambda_i \geq -\varepsilon\|T\|_A\|S\|_A,\]
as desired. Altogether, this shows that given any non-zero scalar $\mu$, there exists a sequence $\{x_n\}\subseteq \mathbb{S}_\mathcal{H}^A$ such that
\[\lim_n \|Tx_n\|_A= \|T\|_A,~\lim_n \left\langle Sx_{n}, Tx_{n}\right\rangle_A=\lambda,~\R~\mathsf{sgn}(\mu)~\lambda \geq -\varepsilon \|T\|_A\|S\|_A.\]
 Now, for $n\in \mathbb{N}$
\begin{align*}
\|T+\mu S\|_A^2 & \geq \|Tx_n+\mu Sx_n\|_A^2\\
& = \|Tx_n\|_A^2+|\mu|^2\|Sx_n\|_A^2+2~\R~\mu \left\langle Sx_{n}, Tx_{n}\right\rangle_A\\
& \geq \|Tx_n\|_A^2 + 2|\mu|~\R~\mathsf{sgn}(\mu)~ \left\langle Sx_{n}, Tx_{n}\right\rangle_A.\\
\end{align*}
Letting $n\to \infty,$ we obtain
\begin{align*}
\|T+\mu S\|_A^2 &\geq \lim\limits_{n\to \infty} \|Tx_n+\mu Sx_n\|_A^2\\
& \geq \lim\limits_{n\to \infty} \big( \|Tx_n\|_A^2 + 2|\mu|~\R~\mathsf{sgn}(\mu)~ \left\langle Sx_{n}, Tx_{n}\right\rangle_A\big).\\
& \geq \|T\|_A^2 - 2\varepsilon  \|T\|_A\|\mu S\|_A.
\end{align*}
In other words, $T\perp_{A,\varepsilon}^B A$, and the proof is completed.
\end{proof}
As an application of Theorem \ref{Theorem: 1}, we state the next example.
\begin{example}
		Let $\mathcal{H}=l^2(\mathbb{N})$ and we consider the following linear operators defined on $\mathcal{H}.$
		\begin{enumerate}
			\item $A(x_1, x_2,\cdots, x_n, \cdots)=(x_1, x_2, 0, x_4,\cdots, x_{n_0}, 0, 0, \cdots)$ where $n_0\in \mathbb{N},$ fixed and $n_0\geq 3,$
				\item $Te_1=\alpha e_1$ and $Te_j=\alpha_{j-1}e_j$, where $\{e_j\}_{j\in \mathbb{N}}$ is the canonical basis of $\mathcal{H}$ and $0<\alpha_1<\cdots<\alpha,$
			\item $Se_1=\frac{\beta}{2}e_1$, $Se_2=\beta e_2$ and $Se_j=e_j$ for any $j\geq 3$ with $1\leq \frac{\beta}{2}<\beta.$
			\end{enumerate}
		Clearly, $A$ is  positive. Also, we have that $\|T\|_A=\alpha=\|Te_1\|_A$ and $\|S\|_A=\beta=\|Se_2\|_A$. \\
		Let $x_n=e_1$ for any $n\in \mathbb{N}$, then we get $\|x_n\|_A=\|e_1\|_A=1, \|Tx_n\|_A=\|Te_1\|_A=\alpha$, and
		\[\langle Sx_n, Tx_n\rangle_A=\langle Se_1, Te_1\rangle_A=\frac{\beta}{2}\alpha,\]
		thus $\frac{\beta}{2}\alpha\in W_A(T, S)$. \\
	Then, for any $\varepsilon \in \left(\frac 12, 1\right)$, we have that $W_A(T, S)\cap B(0,\varepsilon\|T\|_A\|S\|_A)\neq \emptyset$ and so  it follows from Theorem \ref{Theorem: 1} that  $T\bot^B_{A,\varepsilon} S.$
\end{example}
\begin{remark}
The characterization of the Birkhoff-James orthogonality in the context of bounded linear operators obtained by Magajna in \cite{Mag} is a special case of the above result. Notice that our techniques here are completely different  from that of \cite{Mag}.
\end{remark}
\begin{corollary}
Let $T,S\in \mathbb{B}(\mathcal{H})$, then $T\perp_B S$ if and only if there exists a sequence of unit vectors $\{x_n\}$ such that $\|Tx_n\|\to \|T\|$ and $\langle Tx_n, Sx_n\rangle\to 0.$
\end{corollary}
\begin{proof}
Putting $A=I$ and $\varepsilon=0$ in the above theorem, we get
\begin{align*}
T\perp_B S & \iff 0\in\{\lambda \in \mathbb{C}:~\langle Sx_n, Tx_n \rangle \to \lambda,~\{x_n\}\subseteq \mathbb{S}_{\mathcal{H}}^I,~\|Tx_n\|\to \|T\|\}.\\
& \iff 0\in W_I(T, S).
\end{align*}
\end{proof}
	
The characterization of $(\varepsilon,A)$-approximate orthogonality can also be phrased in the following form, without using the set $W_A.$

\begin{theorem}\label{apporthobh}
	Let $T, S\in \mathbb{B}_{A^{1/2}}(\mathcal{H}) $  and $\varepsilon \in [0, 1).$ Then for $T\bot^B_{A,\varepsilon} S$ if and only if for each $\alpha \in [0, 2\pi)$, there exists a sequence $\{x_n^{\alpha}\}\subseteq \mathbb{S}_\mathcal{H}^A=\{x\in \mathcal{H}: \|x\|_A=1 \}$ such that the following conditions hold:
	\begin{itemize}
		\item[(i)] $\lim\limits_{n\to \infty}\|Tx_n^{\alpha}\|_A=\|T\|_A$.
		
		\item[(ii)] $\lim\limits_{n\to \infty}\R(e^{-i\alpha}\langle Tx_n^{\alpha}, Sx_n^{\alpha}  \rangle_A)\geq -\varepsilon \|T\|_A\|S\|_A.$
	\end{itemize}
\end{theorem}
\begin{proof}
Suppose that (i) and (ii) hold. Let $\lambda=|\lambda|e^{i\alpha}\in \mathbb{C}$ with $\alpha \in [0, 2\pi)$. Thus,
\begin{eqnarray}
\|T+\lambda S\|_A^2&\geq &\lim\limits_{n\to \infty}\|(T+\lambda S)x_n^{\alpha}\|_A^2\nonumber \\
&=&\lim\limits_{n\to \infty}(\|Tx_n^{\alpha}\|_A^2+2|\lambda|\R(e^{-i\alpha}\langle Tx_n^{\alpha}, Sx_n^{\alpha}  \rangle_A)+|\lambda|^2\|Sx_n^{\alpha}\|_A^2)\nonumber \\
&\geq& \lim\limits_{n\to \infty}(\|Tx_n^{\alpha}\|_A^2+2|\lambda|\R(e^{-i\alpha}\langle Tx_n^{\alpha}, Sx_n^{\alpha}  \rangle_A)\nonumber\\
&\geq& \|T\|_A^2-2\varepsilon\|T\|_A\|\lambda S\|_A.\nonumber\
\end{eqnarray}
Therefore, $T\bot^B_{A,\varepsilon} S$.
	
Now, we assume that $T\bot^B_{A,\varepsilon} S$ for some $\varepsilon\in [0, 1)$.  Then for any $\alpha \in [0, 2\pi)$, we have $\left\|T+\frac{e^{i\alpha}}{n} S\right\|_A^2\geq \|T\|_A^2-2\varepsilon\|T\|_A\left\|\frac{e^{i\alpha}}{n} S\right\|_A$,
for all $n\in \mathbb{N}$.
	
 Thus, for any $n\in \mathbb{N}$ there exists $x_n^{\alpha}\in \mathbb{S}_\mathcal{H}^A$ such that
\begin{equation*}
\left\|T+\frac{e^{i\alpha}}{n} S\right\|_A^2-\frac{1}{n^2}\leq \left\|\left(T+\frac{e^{i\alpha}}{n} S\right)x_n^{\alpha}\right\|_A^2,
\end{equation*}
 Consequently,
	\begin{eqnarray}\label{normapp}
	\|T\|_A^2&-&\frac{2\varepsilon}{n}\|T\|_A\| S\|_A-\frac{1}{n^2}\leq\left\|T+\frac{e^{i\alpha}}{n
	} S\right\|_A^2-\frac{1}{n^2}\leq\left\|\left(T+\frac{e^{i\alpha}}{n} S\right)x_n^{\alpha}\right\|_A^2\nonumber \\
	&=&\|Tx_n^{\alpha}\|_A^2+\frac2n \R(e^{-i\alpha}\langle Tx_n^{\alpha}, Sx_n^{\alpha}  \rangle_A)+\frac{1}{n^2}\|Sx_n^{\alpha}\|_A^2.  \
	\end{eqnarray}
	Therefore, for all $n\in \mathbb{N}$ we have
	\begin{eqnarray}
	\frac{n}{2}(\|T\|_A^2-\|Tx_n^{\alpha}\|_A^2)\leq \R(e^{-i\alpha}\langle Tx_n^{\alpha}, Sx_n^{\alpha}  \rangle_A)+\frac{1}{2n}\|S\|_A^2+\frac{1}{2n}+\varepsilon \|T\|_A \|S\|_A,\nonumber \
	\end{eqnarray}
	and thus
	\begin{eqnarray}
	0\leq \R(e^{-i\alpha}\langle Tx_n^{\alpha}, Sx_n^{\alpha}  \rangle_A)+\frac{1}{2n}\|S\|_A^2+\frac{1}{2n}+\varepsilon \|T\|_A\|S\|_A. \nonumber \
	\end{eqnarray}
	
 Now, since the sequence $\{\langle Tx_n^{\alpha}, Sx_n^{\alpha}  \rangle_A\}_{n\in \mathbb{N}}$ is bounded in the complex plane,
		 there exists a subsequence $\{\langle Tx_{n_k}^{\alpha}, Sx_{n_k}^{\alpha}  \rangle_A\}_{k\in \mathbb{N}}$  such that
		 \begin{eqnarray*}
	\lim\limits_{k\to \infty}\R(e^{-i\alpha}\langle Tx_{n_k}^{\alpha}, Sx_{n_k}^{\alpha}  \rangle_A)\geq -\varepsilon \|T\|_A \|S\|_A,
	\end{eqnarray*}
	which proves (ii).
	
	Now we shall prove (i). It follows from \eqref{normapp} that for  every natural number $k$,
	\begin{eqnarray}
	\|T\|_A^2&\geq& \|Tx_{n_k}^{\alpha}\|_A^2  \nonumber\\
	&\geq&\|T\|_A^2-\frac{2\varepsilon}{{n_k}}\|T\|_A\| S\|_A-\frac{1}{{n_k}^2}-\frac2{n_k} \R(e^{-i\alpha}\langle Tx_{n_k}^{\alpha}, Sx_{n_k}^{\alpha}  \rangle_A)-\frac{1}{{n_k}^2}\|Sx_{n_k}^{\alpha}\|_A^2\nonumber\\
	&\geq&\|T\|_A^2-\frac{2\varepsilon}{{n_k}}\|T\|_A\| S\|_A-\frac{1}{{n_k}^2}-\frac2{n_k} \|T\|_A \|S\|_A-\frac{1}{{n_k}^2}\|S\|_A^2. \nonumber \
	\end{eqnarray}
	Therefore, $\lim\limits_{k\to \infty}\|Tx_{n_k}^{\alpha}\|_A=\|T\|_A$ and this completes the proof.
\end{proof}

\begin{proposition}
	$T, S\in \mathbb{B}_{A^{1/2}}(\mathcal{H})$ such that  $|\langle Tx_0, Sx_0\rangle_A|\leq \varepsilon \|T\|_A\|S\|_A$ for some $\varepsilon \in [0, 1)$ and $x_0\in \mathbb{M}_T^A\subseteq \mathbb{M}_S^A ,$ then $T\bot^B_{A,\varepsilon} S$ and $S\bot^B_{A,\varepsilon} T.$
\end{proposition}
\begin{proof} By the hypothesis, we have that $\|Tx_0\|_A=\|T\|_A$ and $\|Sx_0\|_A=\|S\|_A.$
		Then, we get
		\[\langle Sx_0, Tx_0\rangle_A \in W_A(T, S)\qquad {\rm and}\qquad |\langle Tx_0, Sx_0\rangle_A|\leq \varepsilon\|T\|_A\|S\|_A.\]
		Therefore, we conclude that $\langle Sx_0, Tx_0\rangle_A \in W_A(T, S)\cap B(0,\varepsilon\|T\|_A\|S\|_A)$ and by Theorem \ref{Theorem: 1}  we obtain $T\bot^B_{A,\varepsilon} S$. In the similar way, $S\bot^B_{A,\varepsilon} T.$
		\end{proof}

\begin{corollary}
Let $S\in \mathbb{B}_{A^{1/2}}(\mathcal{H})$ and $T \in \mathbb{B}_{A}(\mathcal{H})$ such that $\omega_A(T^{\sharp_A}S)<\|S\|_A\|T\|_A,$ then $T\bot^B_{A,\varepsilon} S$ and $S\bot^B_{A,\varepsilon} T$ for some $\varepsilon\in [0,1)$.
\end{corollary}

\begin{proof}By the hypothesis, there exists $\varepsilon \in [0, 1)$ such that $\omega_A(S^{\sharp_A}T)=\varepsilon \|S\|_A\|T\|_A.$ Then,
		for any $x\in \mathcal{H}$ with $\|x\|_A=1$
		\begin{align*}
		 |\langle Tx, Sx  \rangle_A|=|\langle x,T^{\sharp_A}S x  \rangle_A)|\leq \varepsilon\|S\|_A\|T\|_A.
		\end{align*}
		Let $\{x_n\}$  be a sequence in $\mathbb{S}_\mathcal{H}^A$ with $\lim\limits_{n\to \infty}\|Tx_n\|_A=\|T\|_A$.  Now, since the sequence $\{\langle Tx_n, Sx_n  \rangle_A\}_{n\in \mathbb{N}}$ is bounded in the complex plane,  passing through a suitable subsequence, if necessary, we may assume that
		\[\lim\limits_{k\to \infty}\langle Tx_{n_k}, Sx_{n_k} \rangle_A=\lambda,\]
  for some $\lambda \in \mathbb{C}$. Thus, $\bar{\lambda}\in W_A(T, S)\cap B(0,\varepsilon\|T\|_A\|S\|_A)$. Consequently, Theorem \ref{Theorem: 1} ensures that $T\bot^B_{A,\varepsilon} S$. In a similar way, we obtain that  $S\bot^B_{A,\varepsilon} T.$
		\end{proof}
	
 It follows from the inequality \eqref{refine1} that $\omega_A(T) \leq \|T\|_A$ for $T\in \mathbb{B}_{A^{1/2}}(\mathcal{H})$. The next result illustrates that when such inequality is strict we have $(\varepsilon, A)$-approximate orthogonality between $T$ and the identity operator.

\begin{corollary}\label{omegastrictnorm}
	Let $T\in \mathbb{B}_{A^{1/2}}(\mathcal{H})$ such that $\omega_A(T)<\|T\|_A,$ then $T\bot^B_{A,\varepsilon} I$ and $I\bot^B_{A,\varepsilon} T$.
\end{corollary}

\begin{remark}
	The converse of Proposition \ref{omegastrictnorm} is false. Consider $A=I$ and $T=\begin{pmatrix}
	1 & 0\\
	0 & 0
	\end{pmatrix}$. As $T$ is a normal operator, it is well-known that $\omega(T)=\|T\|=1$. On the other hand, for any $\varepsilon \in [0, 1),$ we have that $\frac{\varepsilon}{2}\in W_I(I, T)$. Indeed, if we consider  $x_0=\begin{pmatrix}
	\sqrt{\frac{\varepsilon}{2}} \\
\sqrt{1-\frac{\varepsilon}{2}}
	\end{pmatrix}\in \mathbb{S}_\mathcal{H}^I, $ and $\langle Tx_0, x_0\rangle=\frac{\varepsilon}{2}.$ Then, we conclude that $ W_I(I, T)\cap B(0,\varepsilon)\neq \emptyset$ and by Theorem \ref{Theorem: 1}, we get $I\bot^B_{I,\varepsilon} T$.
\end{remark}
\begin{corollary}
Let	$T, S\in \mathbb{B}_{A^{1/2}}(\mathcal{H})$ such that $x_0\in \mathbb{M}_T^A$ and $\|Sx_0\|_A\leq \varepsilon\|S\|_A$ for some $\varepsilon \in [0,1)$ then $T\bot^B_{A,\varepsilon} S$.
\end{corollary}
\begin{proof}
	By the hyphotesis, we have that $\|Tx_0\|_A=\|T\|_A$ and $\|Sx_0\|_A\leq \varepsilon\|S\|_A$ for some $\varepsilon \in [0,1).$ Then, \[|\langle Tx_0, Sx_0  \rangle_A|\leq \|Tx_0\|_A\|Sx_0\|_A\nonumber\\
	\leq \varepsilon\|T\|_A\|S\|_A,\]
		this means that $\langle  Sx_0, Tx_0,  \rangle_A\in W_A(T, S)\cap B(0,\varepsilon\|T\|_A\|S\|_A)$,  and consequently by Theorem \ref{Theorem: 1}, we obtain $T\bot^B_{A,\varepsilon} S$.
	\end{proof}


\section{More results on approximate $A$-numerical radius orthogonality}\label{s3}
In this section, our objective is to further explore the recent study initiated in \cite{sen} regarding the concept of approximate numerical radius orthogonality and enhance a result presented therein. Let's begin by revisiting the following definition from \cite{sen}.
\begin{definition}\label{omegaorthogonality}
Let $\varepsilon \in [0,1).$ We say that an operator $T\in \mathbb{B}_{A^{1/2}}(\mathcal{H})$ is approximate $(\varepsilon,A)$-numerical radius orthogonal to another element $S\in \mathbb{B}_{A^{1/2}}(\mathcal{H})$ and we write $T\perp^{\varepsilon}_{\omega_A} S$ if
$$\omega_A^2(T+\lambda S)\geq \omega_A^2(T)-2\varepsilon \omega_A(T) \omega_A(\lambda S) \,\, \text{for all }\lambda\in\mathbb{C}.$$
\end{definition}

It is easy to see that the previous notion is homogeneous, i.e. $T\perp^{\varepsilon}_{\omega_A} S$ and $\alpha T\perp^{\varepsilon}_{\omega_A} \beta S$ are equivalent for any $\alpha, \beta \in \mathbb{C}.$ With a proof similar to the one made in Proposition \ref{linearlyind}, we can affirm that if $T\perp^{\varepsilon}_{\omega_A} S,$ then $T, S$ are linearly independent. Next, we give a concrete example of this new notion.
\begin{example}\label{example2}
		Let $\mathcal{H}, A, T$ and $S$  be defined as in Example \ref{example1}. Obviously, $\omega_A(T)=\alpha=\langle Te_1, e_1\rangle_A$ and $\omega_A(S)=1=\langle Te_2, e_2\rangle_A.$
		 \\
		Then, for any $\lambda=|\lambda|e^{i\alpha}\in \mathbb{C}$, we have
		\begin{eqnarray}
		\omega_A^2(T+\lambda S)&\geq &|\langle(T+\lambda S)e_1, e_1\rangle_A|^2\nonumber \\
		&=&|(\alpha+\lambda \varepsilon, 0, 0, \cdots)|^2=|\alpha+\lambda \varepsilon|^2\nonumber\\
		&\geq&(\alpha+\varepsilon|\lambda|\cos(\alpha))^2\nonumber\\
			&\geq&\alpha^2+2\alpha\varepsilon|\lambda|\cos(\alpha)\nonumber\\
		&\geq & \alpha^2-2\alpha\varepsilon|\lambda|=\omega_A^2(T)-2\varepsilon \omega_A(T) \omega_A(\lambda S),\nonumber\
		\end{eqnarray}
		so by Definition \ref{omegaorthogonality}, we conclude that  $T\perp^{\varepsilon}_{\omega_A} S$.
\end{example}

The following theorem, which provides characterizations of approximate $A$-numerical radius orthogonality for $A$-bounded operators, has been proven in \cite{sen}.
\begin{theorem}\label{thms3}
Let $T,S\in \mathbb{B}_{A^{1/2}}(\mathcal{H})$ and $\varepsilon\in [0,1)$. Then the following assertions are equivalent:
\begin{itemize}
  \item [(1)] $T\perp^{\varepsilon}_{\omega_A} S$
    \item [(2)] For each $\alpha\in [0,2\pi)$, there exists a sequence $\{x_k^{\alpha}\}$ in $\mathbb{S}_\mathcal{H}^A$ such that
  \begin{itemize}
    \item [(i)] $\displaystyle\lim_{k\to \infty} |\langle Tx^{\alpha}_k, x^{\alpha}_k\rangle_A| =\omega_A(T)$
    \item [(ii)] $\displaystyle\lim_{k\to \infty} \R\Big(e^{-i\alpha} \langle Tx^{\alpha}_k, x^{\alpha}_k\rangle_A\langle x^{\alpha}_k, Sx^{\alpha}_k\rangle_A\Big)\geq -\varepsilon \omega_A(T)\omega_A(S)$.
  \end{itemize}
\end{itemize}
\end{theorem}

As an application of Theorem \ref{thms3}, we present the following result, which not only removes the additional assumption of $AT=TA$ made by the authors of \cite{sen}, but also provides a characterization, thereby enhancing Theorem 2.10 in \cite{sen}.

\begin{theorem}
Let $T,S\in \mathbb{B}_{A^{1/2}}(\mathcal{H})$ be such that $T\geq_A 0$. Then, the following assertions are equivalent:
\begin{itemize}
  \item [(i)] $T\perp^{\varepsilon}_{\omega_A}S$
  \item [(ii)] $(T+I)\perp^{\varepsilon}_{\omega_A}S$.
  \end{itemize}
\end{theorem}
\begin{proof}
``$(i)\Rightarrow(ii)$'' Let $\alpha \in [0, 2\pi)$. Since $T\perp^{\varepsilon}_{\omega_A}S$, then by Theorem \ref{thms3}, there exists $\{x_k^{\alpha}\}$ in $\mathbb{S}_\mathcal{H}^A$ such that $\displaystyle\lim_{k\to \infty} |\langle Tx^{\alpha}_k, x^{\alpha}_k\rangle_A| =\omega_A(T)$ and
\begin{equation}\label{fr1}
\displaystyle\lim_{k\to \infty} \R\Big(e^{-i\alpha} \langle Tx^{\alpha}_k, x^{\alpha}_k\rangle_A\langle x^{\alpha}_k, Sx^{\alpha}_k\rangle_A\Big)\geq -\epsilon \omega_A(T)\omega_A(S).
\end{equation}
So, by using the fact that $T\geq_A0$, we deduce that
$$\displaystyle\lim_{k\to \infty} \big|\langle Tx^{\alpha}_k, x^{\alpha}_k\rangle_A\big|\R\Big(e^{-i\alpha} \langle x^{\alpha}_k, Sx^{\alpha}_k\rangle_A\Big)\geq -\epsilon \omega_A(T)\omega_A(S),$$
whence
\[\omega_A(T)\lim	_{k\to \infty}\R\big(e^{-i\alpha}\langle x^{\alpha}_k, Sx^{\alpha}_k\rangle_A\big)
\geq -\varepsilon \omega_A(T)\omega_A(S).\]
 This yields that
 \begin{equation}\label{fr2}
\lim	_{k\to \infty}\R\big(e^{-i\alpha}\langle x^{\alpha}_k, Sx^{\alpha}_k\rangle_A\big)
\geq -\varepsilon \omega_A(S).
\end{equation}
Moreover, since $T\geq_A0$, then it can be seen that $\omega_A(T+I)=\omega_A(T)+1$. So, by taking \eqref{fr1} and \eqref{fr2} into consideration, one observes that
\begin{align*}
&\lim_{k\to\infty} \R\big(e^{-i\alpha}\langle (T+I) x_k^{\alpha},x_k^{\alpha}\rangle_A\langle x_k^{\alpha},Sx_k^{\alpha}\rangle_A\big)\\
&=\lim_{k\to\infty}\R\big(e^{-i\alpha}\langle Tx_k^{\alpha},x_k^{\alpha}\rangle_A \langle x^{\alpha}_k, Sx^{\alpha}_k\rangle_A\big)+\lim_{k\to\infty} \R\big(e^{-i\alpha}\|x_k^{\alpha}\|_A^2\langle x_k^{\alpha},Sx_k^{\alpha}\rangle_A\big)\\
&\geq -\varepsilon \omega_A(T) \omega_A(S)-\varepsilon \omega_A(S)\\
&= -\varepsilon\omega_A(S)\big(\omega_A(T)+1\big)\geq -\varepsilon\omega_A(S)\omega_A(T+I).
\end{align*}
On the other hand, we have
\begin{align*}
\lim_{k\to \infty}\big|\langle (T+I) x_k^{\alpha}, x_k^{\alpha}\rangle_A\big|^2
&=\lim_{k\to \infty}\left(\big|\langle T x_k^{\alpha},x_k^{\alpha}\rangle_A\big|^2+\|x_k^{\alpha}\|_A^2+2\langle Tx^{\alpha}_k, x^{\alpha}_k\rangle_A\right)\\
&=\omega_A^2(T)+1+2\omega_A(T)=\omega_A^2(T+I).
\end{align*}
Thus, we infer that $(T+I)\perp ^\varepsilon_{\omega_A} S$ as required.

``$(ii)\Rightarrow(i)$'' Let $\alpha \in [0, 2\pi)$. Since $(T+I)\perp^{\varepsilon}_{\omega_A}S$, then by Theorem \ref{thms3}, there exists $\{x_k^{\alpha}\}$ in $\mathbb{S}_\mathcal{H}^A$ such that
\begin{equation}\label{fr3}
\displaystyle\lim_{k\to \infty} |\langle (T+I)x^{\alpha}_k, x^{\alpha}_k\rangle_A| =\omega_A(T+I),
\end{equation}
and
\begin{equation}\label{fr4}
\displaystyle\lim_{k\to \infty} \R\Big(e^{-i\alpha} \langle (T+I)x^{\alpha}_k, x^{\alpha}_k\rangle_A\langle x^{\alpha}_k, Sx^{\alpha}_k\rangle_A\Big)\geq -\epsilon \omega_A(T+I)\omega_A(S).
\end{equation}
Since $T\geq_A 0$, then $\omega_A(T+I)=\omega_A(T)+1$. So, it follows from \eqref{fr3} that $\displaystyle\lim_{k\to \infty} |\langle Tx^{\alpha}_k, x^{\alpha}_k\rangle_A| =\omega_A(T)$. Moreover, it follows from \eqref{fr4} that
\begin{equation*}
\displaystyle\lim_{k\to \infty} \R\Big(e^{-i\alpha} \big(\langle Tx^{\alpha}_k, x^{\alpha}_k\rangle_A+1\big)\langle x^{\alpha}_k, Sx^{\alpha}_k\rangle_A\Big)\geq -\epsilon \big(\omega_A(T)+1\big)\omega_A(S),
\end{equation*}
whence
\begin{equation}\label{5ra}
\displaystyle\lim_{k\to \infty}\left[\big(\langle Tx^{\alpha}_k, x^{\alpha}_k\rangle_A+1\big) \R\Big(e^{-i\alpha} \langle x^{\alpha}_k, Sx^{\alpha}_k\rangle_A\Big)\right]\geq -\epsilon \big(\omega_A(T)+1\big)\omega_A(S),
\end{equation}
 Since \{$\R\Big(e^{-i\alpha} \langle x^{\alpha}_k, Sx^{\alpha}_k\rangle_A\Big)\}$ is a bounded sequence in $\mathbb{C}$, there exists a subsequence $\{y_{k_j}^{\alpha}\}$ of $\{x_k^{\alpha}\}$ such that
$$\displaystyle\lim_{j\to\infty}\R\Big(e^{-i\alpha} \langle y_{k_j}^{\alpha}, Sy_{k_j}^{\alpha}\rangle_A\Big)=M<\infty.$$
So, by using \eqref{5ra}, we deduce that
\begin{equation*}\label{5ra2}
(\omega_A(T)+1\big)M\geq -\epsilon \big(\omega_A(T)+1\big)\omega_A(S),
\end{equation*}
This yields that $M\geq -\epsilon \omega_A(S)$. Now, we see that
\begin{align*}
\displaystyle\lim_{j\to \infty} \R\Big(e^{-i\alpha} \langle Ty_{k_j}, y_{k_j}\rangle_A\langle y_{k_j}, Sy_{k_j}\rangle_A\Big)
&=\displaystyle\lim_{j\to \infty} \langle Ty_{k_j}, y_{k_j}\rangle_A\R\Big(e^{-i\alpha} \langle y_{k_j}, Sy_{k_j}\rangle_A\Big)\\
&=\omega_A(T)\cdot M\geq -\epsilon \omega_A(T)\omega_A(S).
\end{align*}
This proves that $T\perp^{\varepsilon}_{\omega_A}S$.
\end{proof}

The following example shows that the $A$-positivity condition is required in the last Theorem.
	\begin{example}
		Suppose that $A=I,$ $T=\begin{pmatrix}
		-1 & 1\\
		0 & -2
		\end{pmatrix}$ and $S=\begin{pmatrix}
		\frac{1}{1000} & 0\\
		0 & 0
		\end{pmatrix}$ are in $\mathbb{M}_2(\mathbb{C}).$  Clearly, $T$ is not a positive operator, since $\langle Te_1, e_1\rangle=-1<0.$\\
		Also, we have $|\langle Te_1, e_1\rangle|=1$ and for any $\alpha\in [0, 2\pi)$, we get
		\begin{align*}
		\R\big(e^{-i\alpha}\langle Te_1,e_1\rangle \langle e_1, Se_1\rangle\big)
		&=	\R\Big(e^{-i\alpha}\frac{-1}{1000}\Big)\\
		&=\frac{-1}{1000}\R\big(e^{-i\alpha}\big)=\frac{-1}{1000}\cos(\alpha)\\
		&\geq \frac{-1}{1000}\geq -\varepsilon,\
		\end{align*}	
		for each $\varepsilon \in [\frac{1}{1000}, 1).$ Therefore, due to the homogeneity of the $\perp^{\varepsilon}_{\omega_A}$, we conclude that $T\perp^{\varepsilon}_{\omega} S$ for any $\varepsilon \in [\frac{1}{1000}, 1)
		.$\\
		On the other hand, for any $\lambda \in \mathbb{C}$ we get
		\[T+I+\lambda S=\begin{pmatrix}
		\frac{\lambda}{1000} & 1\\
		0 & -1
		\end{pmatrix}.\]
		In particular, for $\lambda=1000$, we obtain $\omega\Bigg(\begin{pmatrix}
		\frac{\lambda}{1000} & 1\\
		0 & -1
		\end{pmatrix}\Bigg)=\frac{\sqrt{5}}{2}$ and $\omega(T+I)=\omega\Bigg(\begin{pmatrix}
		0 & 1\\
		0 & -1
		\end{pmatrix}\Bigg)=\frac{1+\sqrt{2}}{2}.$\\
		 Then, if we consider $\varepsilon_0=\frac{1}{100}\in [\frac{1}{1000}, 1)$, we reach
		\[ 1.24\approx \omega^2(T+I+1000 S)<\omega^2(T+I)-2\varepsilon_0\omega(T+I)\omega(1000 S)\approx 1.43,\]
		showing that $ (T+I)\not\perp^{\varepsilon_0}_{\omega}  S  $ where $\varepsilon_0=\frac{1}{100}.$	
\end{example}

\section*{Declarations}
\noindent{\bf{Funding}}\\
\noindent No applicable.

\vspace{0.5cm}

\noindent{\bf{Availability of data and materials}}\\
\noindent No data were used to support this study.
\vspace{0.5cm}\\
\noindent{\bf{Competing interests}}\\
\noindent The authors declare that they have no competing interests.
\vspace{0.5cm}

\noindent{\bf{Author contribution}}\\
\noindent The work presented here was carried out in collaboration between all authors. All authors contributed equally and significantly in writing this article. All authors have contributed to the manuscript. All authors have read and agreed to the published version of the manuscript.

\vspace{0.5cm}

\noindent{\bf Acknowledgment:} The authors would like to express their gratitude to Saikat Roy from the Indian Institute of Technology Bombay for his valuable assistance in the preparation of this work. Additionally, some of the results presented in this paper have been privately discussed with him.

\bibliographystyle{amsplain}

\end{document}